\documentclass[12pt]{article}
\usepackage{fancyhdr, array,calc,graphicx,url,tabularx}
\usepackage{geometry}
\usepackage{latexsym}
\usepackage{amssymb}
\usepackage{amsmath}
\usepackage{makeidx}
\usepackage{undertilde}

\makeindex

{
   \end{minipage}
   \vspace*{\stretch{3}}
   \clearpage
}

\renewcommand{\gg}{\gamma}

\newcommand{\rest}{\restriction}

\renewcommand{\models}{\vDash}

\newtheorem{theorem}{Theorem}[section]

\newtheorem{definition}[theorem]{Definition}

\newtheorem{lemma}[theorem]{Lemma}

\numberwithin{figure}{section}

\newenvironment{proof}{{\it{
Proof.}}}{\nopagebreak\mbox{}{\hfill$\square$}
\par\bigskip}
\newenvironment{Claim}{\noindent{\it
Claim} {}}{}
\newenvironment{Case}{\noindent{\it
Case}}{}

\newcommand{\rthm}[1]{Theorem~\ref{#1}}
\newcommand{\rlem}[1]{Lemma~\ref{#1}}

\newcommand{\rdef}[1]{Definition~\ref{#1}}

\def\k{\kappa}
\def\a{\alpha}
\def\b{\beta}
\def\d{\delta}

\def\l{\lambda}

\def\P{{\mathcal{P} }}
\def\W{{\mathcal{W} }}
\def\Q{{\mathcal{ Q}}}

\def\R{{\mathcal R}}

\def\H{{\rm{HOD}}}
\def\M{{\mathcal{M}}}

\def\T {{\mathcal{T}}}
\def\U{{\mathcal{U}}}
\def\S{{\mathcal{S}}}

\def\iff{\mathrel{\leftrightarrow}}

\def\and{\mathrel{\kern1pt\&\kern1pt}}

\def\<#1>{\langle\,#1\,\rangle}

\def\inseg{\trianglelefteq}

 \input xy
 \xyoption{all}

\title{An inner model proof of the strong partition property for $\utilde{\delta}^2_1$ 
}
\author{Grigor Sargsyan \thanks{This material is partially based upon work supported by the National Science Foundation under Grant No DMS-0902628. Part of this paper was written while the author was a Leibniz Fellow at the Mathematisches Forschungsinstitut Oberwolfach.}\\
        Department of Mathematics\\
        Rutgers University\\
        Hill Center for the Mathematical Sciences\\
        110 Frelinghuysen Rd.\\
        Piscataway, NJ 08854 USA\\
        http://math.rutgers.edu/$\sim$gs481\\
        grigor@math.rutgers.edu}
\date{ \today}

\begin{document}

\maketitle

\begin{abstract}
Assuming $V=L(\mathbb{R})+AD$, using methods from inner model theory, we give a new proof of the strong partition property for $\utilde{\d}^2_1$. The result was originally proved in \cite{KKMW}.
\end{abstract}

The main theorem of this note is the following special case of Theorem 1.1 of \cite{KKMW} originally due to Kechris-Kleinberg-Moschovakis-Woodin. 

\begin{theorem}\label{main theorem} Assume $V=L(\mathbb{R})+AD$. Then $\utilde{\delta}^2_1$ has the strong partition property, i.e., $\utilde{\delta}^2_1\rightarrow (\utilde{\delta}^2_1)^{\utilde{\delta}^2_1}$ holds.
\end{theorem}
 
Our proof uses techniques from inner model theory and resembles Martin's proof of strong partition property for $\omega_1$ (see \cite{Jackson}). We expect that it will have other applications and in particular, can be used to show that under $AD^+$, if $\Gamma$ is any $\Pi^1_1$-like \footnote{i.e., closed under $\forall^{\mathbb{R}}$ and non-selfdual} scaled pointclass and $\d=\d(\Gamma)$ then $\d$ has the strong partition property. Our motivation to find a new proof of \rthm{main theorem} comes from a desire to prove Kechris-Martin like results for $\Pi^1_1$-like scaled pointclasses which will settle Question 19 of \cite{OpenProblems} and most likely, several other questions in the same neighborhood. We are optimistic that inner model theoretic techniques will settle this question and our optimism comes from the fact that the literature is already full of descriptive set theoretic results that have been proved using methods from inner model theory (for instance, see \cite{Hjorth01}, \cite{GeneralizedHjorth} and \cite{OIMT}). More importantly for us, recently, Neeman, in \cite{KMN}, found a proof of the Kechris-Martin theorem for $\Pi^1_3$ using techniques from inner model theory. Finally, we believe that our proof can be used to prove the strong partition property for many cardinals $\d=\d(\Gamma)$ where $\Gamma$ has strong closure properties. In fact, we expect that it can be used to prove Theorem 1.1 of \cite{KKMW} but we certainly haven't done so. We now start proving \rthm{main theorem}.\\

\begin{proof}
Let $\kappa=\utilde{\delta}^2_1$. By Martin's theorem (see Theorem 2.31 and Definition 2.30 of \cite{Jackson}), it is enough to show that $\k$ is $\k$-reasonable, i.e., there is a non-selfdual pointclass $\utilde{\Gamma}$ closed under $\exists^{\mathbb{R}}$ and a map $\phi$ with domain $\mathbb{R}$ satisfying:
\begin{enumerate}
\item $\forall x (\phi(x)\subseteq \k\times \k)$,
\item $\forall F: \k\rightarrow \k$, $\exists x \in \mathbb{R} ( \phi(x)=F)$,
\item $\forall \b<\k$, $\forall \gg<\k$, $R_{\b, \gg}\in \utilde{\Delta}$ where 
\begin{center}
$x\in R_{\b, \gg} \iff \phi(x)(\b, \gg) \wedge \forall \gg^\prime <\k (\phi(x)(\b, \gg^\prime) \rightarrow \gg^\prime=\gg)$
\end{center}
\item Suppose $\b<\l$, $A\in \exists^{\mathbb{R}}\utilde{\Delta}$, and $A\subseteq R_\b=\{ x : \exists \gg<\k R_{\b, \gg}(x)\}$. Then $\exists \gg_0<\k$ such that $\forall x\in A\exists \gg<\gg_0 R_{\b, \gg}(x)$. 
\end{enumerate}

Let $\Gamma=\Sigma^2_1$. We claim that $\utilde{\Gamma}$ is as desired and spend the rest of the proof to argue for it. In what follows, we will freely use the terminology developed for analyzing $\H$ of models of $AD^+$. This terminology has been exposited in many places including \cite{GeneralizedHjorth}, \cite{StrengthPFA1}, \cite{ATHM},  \cite{CMI}, \cite{OIMT} and more recently in \cite{SSW}. In particular, recall the definitions of suitable premouse, short tree, maximal tree, short tree iterable and etc. Given a suitable premouse $\P$, we let $\d_\P$ be its Woodin cardinal and $\l_\P$ be the least cardinal which is $<\d_\P$-strong in $\P$.

Suppose $a\in HC$. We say an $a$-premouse $\Q$ is \textit{good} if 
\begin{enumerate}
\item $\Q$ is $(\omega, \omega_1)$-iterable,
\item $\Q\models ZFC-Powerset$+``there are no Woodin cardinals" +``there is a largest cardinal"
\item $\Q$ is full, i.e., for every cutpoint $\xi$ of $\Q$, $Lp(\Q|\xi)\inseg\Q$. 
\end{enumerate} 
If $\Q$ is good then it has a unique $(\omega, \omega_1)$-iteration strategy with Dodd-Jensen property. We let $\Sigma_\Q$ be this strategy. Also, let $\eta_\Q$ be the largest cardinal of $\Q$. Given an iteration tree $\T$ on $\Q$ according to $\Sigma_{\Q}$ with last model $\R$ such that $\pi^{\T}$ exists, we let $\pi_{\Q, \R}:\Q\rightarrow \R$ be the iteration embedding. Notice that because $\Sigma_\Q$ has the Dodd-Jensen property, $\pi^\T$ is independent of $\T$. We say $\Q$ is \textit{excellent} if whenever $\R$ is a $\Sigma_\Q$-iterate of $\Q$ such that $\pi_{\Q, \R}$ is defined $\R$ is good. In this case, we also say that $\Sigma_\Q$ is fullness preserving.

Suppose now $\a<\k$ is such that it ends a weak gap (see \cite{Scales}). We then let 
\begin{center}
$\mathcal{F}(\a, a)=\{ \Q: J_\a(\mathbb{R})\models ``\Q$ is an excellent $a$-premouse"$\}$.
\end{center} 
Given $a$-premouse $\P$ such that $J_\a(\mathbb{R})\models ``\P$ is suitable and short tree iterable" we let
$\mathcal{F}(\a, a, \P)$ be the set of $\Q$ such that in $J_\a(\mathbb{R})$, there is a correctly guided short tree $\T$ on $\P$ with last suitable model $\P^*$ such that for some $\P^*$-cardinal $\eta\leq \l_{\P^*}$, $\Q=\P^*|(\l_{\P^*}^+)^{\P^*}$.  

\begin{lemma} Suppose $\a<\k$ ends a weak gap, $a\in HC$ and $\P$ is an $a$-premouse such that $J_\a(\mathbb{R})\models ``\P$ is suitable and short tree iterable". Then $\mathcal{F}(\a, a, \P)\subseteq \mathcal{F}(\a, a)$.
\end{lemma}
\begin{proof}
Fix $\Q\in \mathcal{F}(\a, a, \P)$. Work in $J_\a(\mathbb{R})$. Let $\T$ be a correctly guided short tree on $\P$ with last suitable model $\P^*$ such that for some $\P^*$-cardinal $\eta< \l_{\P^*}$, $\Q=\P^*|(\l_{\P^*}^+)^{\P^*}$. Because $\P$ is short tree iterable, we have that $\Q$ is $(\omega, \omega_1)$-iterable via a unique iteration strategy $\Sigma$. As the iterations of $\Q$ can also be viewed as iterations of $\P^*$, we have that $\Sigma$ is fullness preserving, implying that $\Q$ is excellent.
\end{proof}

Notice that if $\b>\a$ is such that $\b$ ends a weak gap and $J_\b(\mathbb{R})\models ``\P$ is suitable and short tree iterable $a$-premouse", then there could be $\Q\in \mathcal{F}(\b, a, \P)$ which is not in $\mathcal{F}(\a, a, \P)$. However, we always have the following easy lemma.

\begin{lemma}\label{inclusion} Suppose $\a<\b<\k$ are two ordinals which end weak gaps and such that $J_\a(\mathbb{R})$ and $J_\b(\mathbb{R})$ both satisfy that $\P$ is suitable and short tree iterable. Then $\mathcal{F}(\a, a, \P)\subseteq \mathcal{F}(\b, a, \P)$.
\end{lemma}
\begin{proof}
The lemma follows because any iteration tree on $\P$ which is correctly guided and short in the sense of $J_\a(\mathbb{R})$ is also correctly guided and short in the sense of $J_\b(\mathbb{R})$.
\end{proof}

Next we define $\leq_{\a, a}$ on $\mathcal{F}(\a, a)$ by setting $\Q\leq_{\a, a}\R$ iff there is an iteration tree $\T$ on $\Q$ according to $\Sigma_\Q$ with last model $\S$ such that $\pi^\T$ exists, $\S\inseg\R$ and $\S=\R|(\eta_\S^+)^\R$. Also, let $\leq_{\a, a, \P}=\leq_{\a, a}\rest \mathcal{F}(\a, a)$. As usual, we have that

\begin{lemma}\label{directedness} $\leq_{\a, a}$ and $\leq_{\a, a, \P}$ are directed, and $\leq_{\a, a, \P}$ is dense in $\leq_{\a, a}$.
\end{lemma}

Let then $\M_\infty(\a, a)$ be the direct limit of $(\mathcal{F}(\a, a), \leq_{\a, a})$ under the iteration embeddings $\pi_{\Q, \R}$. Also, let $\M_\infty(\a, a, \P)$ be the direct limit of $(\mathcal{F}(\a, a, \P), \leq_{\a, a, \P})$ under the iteration embeddings $\pi_{\Q, \R}$. It follows from \rlem{directedness} that

\begin{lemma}\label{equality of direct limits} $\M_\infty(\a, a)=\M_\infty(\a, a, \P)$.
\end{lemma}
We let $\pi_{\Q, \infty}:\Q\rightarrow \Q^*\inseg\M_\infty(\a, a, \P)$ be the direct limit embedding\footnote{We drop $\a$ and $a$ from our notation as the embedding doesn't depend on them}. 

We can now define $\phi$. First let $S$ be the set of those reals $x$ which code a pair $(y_x, \P_x)$ such that 
\begin{enumerate}
\item $y_x\in \mathbb{R}$,
\item for some $\a<\k$ ending a weak gap, $J_\a(\mathbb{R})\models ``\P_x$ is suitable and short tree iterable $y_x$-premouse". 
\end{enumerate} 
Clearly $S$ is $\Sigma^2_1$. Also let $f:\kappa^2\rightarrow \k$ be the function given by: for all $(\b, \gg)\in \kappa^2$, $f(\b, \gg)$ is the least ordinal $\a$ such that $\alpha$ ends a weak gap and $J_\a(\mathbb{R})\models \max(\b, \gg)<\utilde{\d}^2_1$. Notice that $f$ is $\Delta^2_1$ in codes. We define $\phi$ as follows. 
\begin{definition}\label{definition of phi} If 
$x\not \in S \cap \mathbb{R}$ then let $\phi(x)=\emptyset$. Suppose now $x\in S$. Let $(y_x, \P_x)$ be the pair coded by $x$.  Given $\b, \gg<\k$, we let $(\b, \gg)\in \phi(x)$ iff letting $\P=\P_x$ and $f(\b, \gg)=\a$ then for some $a\in \P$ the following holds in $J_\a(\mathbb{R})$:
\begin{enumerate}
\item $\P$ is suitable and short tree iterable, 
\item $a$ is the collapse of $x(0)$, 
\item $a\subseteq \l_\P\times \l_\P$,
\item there is a correctly guided short tree $\T$ with last model $\S$ such that $\pi_{\P, \S}$ exists and an $\S$-cardinal $\eta$ such that
\begin{enumerate}
\item $(\eta^+)^\S<\l^\S$,
\item if $\Q=\S|(\eta^+)^\S$ and $a^\Q=\pi_{\P, \S}(a)\rest \eta$ then $(\b, \gg)\in \pi_{\Q, \infty}(a^\Q)\cap rng(\pi_{\Q, \infty})$.
\end{enumerate} 
\end{enumerate}
\end{definition}
Given $\a<\Theta$ we let $S_a$ and $\phi_\a$ be what the above definitions give over $J_\a(\mathbb{R})$. 
The following lemmas establish that $\phi$ is as desired. We start with the following easy lemma.

\begin{lemma}\label{easy lemma} For each $x\in \mathbb{R}$, $\phi(x)=\cup_{\a<\k} \phi_\a(x)$. 
\end{lemma}
\begin{proof} Suppose $(\b, \gg)\in \phi(x)$. Then letting $\a=f(max(\b, \gg), f(\b, \gg))$. Then $(\b, \gg)\in \phi_\a(x)$. The other direction is similar.
\end{proof}

\begin{lemma} For every $x\in \mathbb{R}$, $\phi(x)\subseteq \kappa\times \kappa$.
\end{lemma} 
\begin{proof} The claim follows from the fact that for every $\a$ and $a$, $\M_{\infty}(\a, a)\subseteq J_\a(\mathbb{R})$. 
\end{proof}

\begin{lemma} Suppose $F:\kappa\rightarrow \kappa$. Then there is $x\in dom(\phi)$ such that $\phi(x)=F$. 
\end{lemma}  
\begin{proof}
Fix $y$ such that $F\in \H_y$. There is then a suitable $\P$ over $y$ such that $F\in rng(\pi_{\P, \emptyset, \infty})$\footnote{Recall the direct limit construction that converges to $\H|\Theta$. Here $\pi_{\P, \emptyset, \infty}$ is the direct limit embedding given by $\emptyset$-iterability embeddings. For more details see either of the aforementioned papers.} Notice that $\pi_{\P, \emptyset, \infty}(\l_\P)=\k$ (see Chapter 8 of \cite{OIMT}). Let then $a\subseteq \l_\P\times \l_\P$ be such that $\pi_{\P, \emptyset, \infty}(a)=F$ and let $x$ code the pair $(y, \P)$ such that $x(0)=a$. It is then easy to see that $\phi(x)=F$ (use \rlem{easy lemma}). 
\end{proof}

\begin{lemma} Suppose $\b, \gg<\k$. Let 
\begin{center}
$x\in R_{\b, \gg} \iff \phi(x)(\b, \gg) \wedge \forall \gg^\prime <\k (\phi(x)(\b, \gg^\prime) \rightarrow \gg^\prime=\gg)$.
\end{center}
Then $R_{\b, \gg}$ is $\utilde{\Delta}^2_1$.
\end{lemma}
\begin{proof}
We have that the following are equivalent:
\begin{enumerate}
\item $x\in R_{\b, \gg}$.
\item  There is $\a>f(\b, \gg)$ such that $J_{\a}(\mathbb{R})\models ``x\in dom(\phi_\a)$ and $\gg$ is the unique ordinal such that $(\b, \gg)\in \phi_\a(x)"$,
\item For all $\a>f(\b, \gg)$, $J_{\a}(\mathbb{R})\models ``x\in dom(\phi_\a)$ and $\gg$ is the unique ordinal such that $(\b, \gg)\in \phi_\a(x)"$
\end{enumerate}

Clearly 1 implies 2 and 3. Also, that 3 implies 1 is rather straightforward. We show that 2 implies 1. Fix then $\a>f(\b, \gg)$ such that $J_{\a}(\mathbb{R})\models ``x\in dom(\phi_\a)$ and $\gg$ is the unique ordinal such that $(\b, \gg)\in \phi_\a(x)"$. Let $(y, \P)$ be the pair coded by $x$ and $a\in \P$ the transitive collapse of $x(0)$. Working in $J_\a(\mathbb{R})$, let $\T$ be a correctly guided short tree on $\P$ with last model $\S$ such that $\pi_{\P, \S}$ exists and an $\S$-cardinal $\eta$ such that
\begin{enumerate}
\item $(\eta^+)^\S<\l^\S$,
\item if $\Q=\S|(\eta^+)^\S$ and $a^\Q=\pi_{\P, \S}(a)\rest \eta$ then $(\b, \gg)\in \pi_{\Q, \infty}(a^\Q)\cap rng(\pi_{\Q, \infty})$.
\end{enumerate} 
Suppose now there is some $\xi$ such that for some $\gg^\prime$, $(\b, \gg^\prime)\in \phi_\xi(x)$. Working in $J_\xi(\mathbb{R})$, let $\T^*$ be a correctly guided short tree on $\P$ with last model $\S^*$ such that $\pi_{\P, \S}$ exists and an $\S^*$-cardinal $\nu$ such that
\begin{enumerate}
\item $(\nu^+)^\S<\l^\S$,
\item if $\R=\S|(\nu^+)^\S$ and $a^\R=\pi_{\P, \S^*}(a)\rest \nu$ then $(\b, \gg^\prime)\in \pi_{\R, \infty}(a^\R)\cap rng(\pi_{\R, \infty})$.
\end{enumerate} 
Without loss of generality assume that $\xi>\a$. We then have that $J_\xi(\mathbb{R})\models ``\S$ and $\S^*$ are suitable and short tree iterable". Work now in $J_\xi(\mathbb{R})$. We can then find $\S^{**}$ which is a suitable correct iterate of both $\S$ and $\S^*$. Notice that since $\S^{**}$ is suitable, the iteration embeddings $i : \S|(\l_\S^+)^\S\rightarrow \S^{**}|(\l_{\S^{**}}^+)^{\S^{**}}$ and $j: \S^*|(\l_{\S^*}^+)^{\S^*}\rightarrow \S^{**}|(\l_{\S^{**}}^+)^{\S^{**}}$ exists. 

Suppose now that $\gg\not =\gg^\prime$. Let $(\bar{\b}, \bar{\gg}, \bar{\gg}^\prime)\in \S^{**}$ be such that letting $\zeta=max(i(\eta_\Q), j(\eta_\R))$ and $\W=\S^{**}|(\zeta^+)^{\S^{**}}$, $\pi_{\W, \infty}(\bar{\b}, \bar{\gg}, \bar{\gg}^\prime)=(\b, \gg, \gg^\prime)$. It then follows that $(\bar{\b}, \bar{\gg})\in i(\pi^\T_{\P, \S}(a))$ and $(\bar{\b}, \bar{\gg}^\prime)\in j(\pi^{\T^*}_{\P, \S^*}(a))$. However, $i\circ \pi^\T_{\P, \S}=j\circ \pi^{\T^*}_{\P, \S^*}$, implying that $i(\pi^\T_{\P, \S}(a))=j(\pi^{\T^*}_{\P, \S^*}(a))$ and that $S^{**}\models (\bar{b}, \bar{\gg})\in i(\pi^\T_{\P, \S}(a)) \wedge (\bar{b}, \bar{\gg}^\prime)\in i(\pi^\T_{\P, \S}(a))$.

Let now $(\tau, \tau^*)\in \Q$ be such that $\pi_{\Q, \infty}(\tau, \tau^*)=(\b, \gg)$. By elementarity of $i$, we then get that $\S\models ``$there is $\tau^{**}\not =\tau^*$ such that $(\tau, \tau^{**})\in \pi_{\P, \S}(a)"$. Fix such a $\tau^{**}$ and let $\varsigma \in (\tau^{**}, \l_\S)$ be an $\S$-cardinal. Then letting $\Q^*=\S|(\varsigma^+)^\S$ we have that $(\b, \pi
_{\Q^*, \infty}(\tau^{**}))\in \phi_\a(x)$ and $\pi_{\Q^*, \infty}(\tau^{**})\not = \gg$, contradiction. 
\end{proof}

The next lemma finishes the proof.

\begin{lemma} Suppose $\b<\l$, $A\in \utilde{\Delta}^2_1$ and $A\subseteq R_\b=\{ x : \exists \gg<\k R_{\b, \gg}(x)\}$. Then $\exists \gg_0<\k$ such that $\forall x\in A\exists \gg<\gg_0 R_{\b, \gg}(x)$. 
\end{lemma}
\begin{proof}
Let $f:A\rightarrow \k$ be defined by $f(x)=\nu$ if $\nu$ is the least such that $\nu$ ends a weak gap and $J_\nu(\mathbb{R})\models x\in R_\b$. Then $f$ is $\utilde{\Sigma}_1$ over $J_{\k}(\mathbb{R})$ and hence, as $\k$ is $\mathbb{R}$-admissible, $f$ is bounded.
\end{proof}
\end{proof}

\bibliographystyle{plain}
\bibliography{pertitionproperty}

\end{document}